\documentclass[12pt, reqno]{amsart}
\usepackage{latexsym, amsmath, amssymb, amsthm, mathscinet, mathtools, enumerate}
\usepackage{cases, verbatim}
\usepackage[active]{srcltx}
\usepackage{hyperref}
\usepackage{xcolor}
\usepackage{enumitem}

\usepackage[margin=1.5in]{geometry}

\def\R{\mathbb R}

\def\G{\mathbb G}
\def\g{\mathfrak{g}}

\def\X{\mathrm X}

\def\dim{\mathrm {dim}}

\newtheorem{theorem}{Theorem}
\newtheorem{lemma}{Lemma}
\newtheorem{remark}{Remark}

\newtheorem{proposition}{Proposition}

\title[Step 2 Carnot groups are UD]{Left-invariant metrics on step-two Carnot groups are uniformly doubling}               

\author[C. Bi]{Cheng Bi}
\address[C. Bi]{Chern Institute of Mathematics, Nankai University, Tianjin 300071, P. R. China, {\tt cbi24158@gmail.com}}

\author[Y. Zhang]{Ye Zhang}
\address[Y.~Zhang]{SISSA, via Bonomea 265, 34136 Trieste, Italy, {\tt yezhang@sissa.it}}

\begin{document}

 \maketitle

\vspace{-1.0cm}

\begin{abstract}
We prove that the uniform doubling property holds for left-invariant Riemannian metrics on every step-two Carnot group and the sharp uniform doubling constant is $2^Q$, where $Q$ is the homogeneous dimension. More generally, every such metric satisfies the Bishop--Gromov type estimate.
\end{abstract}

\vspace{1.0cm}

\section{Introduction}

\medskip

The uniform doubling property is an interesting phenomenon on Lie groups. Its study was initiated by Eldredge, Gordina, and Saloff-Coste \cite{EGS18}. Let \(\G \) be a connected Lie group, and let \(\mathcal{L}(\G)\) denote the set of all left-invariant Riemannian metrics on \(\G\). Write \(\nu_g\) for the Riemannian volume induced by $g\in \mathcal{L}(\G)$, and $B_g (p,r)$ for  the open ball of the Riemannian distance induced by $g \in \mathcal{L}(\G)$ centered at $p \in \G$ with radius $r > 0$. We say $\G$ is uniform doubling if
\begin{equation*}
    D(\G) \ := \sup_{g \in \mathcal{L}(\G), \, p \in \G, \, r > 0} \frac{\nu_g(B_g(p,2r))}{\nu_g(B_g(p,r))} <+ \infty. 
\end{equation*}
The authors of \cite{EGS18} indeed conjectured that every connected compact Lie group is uniform doubling. To the best of authors' knowledge, the nontrivial examples only include torus, $\mathrm{SU}(2)$, $\mathrm{SO}(3)$ and more generally, quotient groups of $\mathrm{SU}(2) \times \R^n$ \cite{EGS18,EGS24}. Their proofs rely on an explicit formula describing of the growth of the volume.

\medskip

The purpose of this paper is to investigate the uniform doubling property in the setting of step-two Carnot groups. These groups, including the famous Heisenberg groups, form a broad class of nilpotent Lie groups and serve as fundamental model spaces in sub-Riemannian geometry since they appear naturally as tangent spaces of sub-Riemannian manifolds. Our main theorem is as follows.
\begin{theorem}\label{t1}
    Assume $\G$ is a step-two Carnot group. Let $Q$ be the homogeneous dimension of $\G$ (see \eqref{hd}). Then $D(\G) = 2^Q$.
\end{theorem}

Note that we not only prove that every step-two Carnot group is uniform doubling but also find the exact value of $D(\G)$. Recall that the family of left-invariant metrics under consideration has no uniform Ricci lower bound, otherwise as measured Gromov--Hausdorff limits of the left-invariant Riemannian metrics, the left-invariant sub-Riemannian metrics will satisfy some curvature-dimension condition, but it is not the case, see \cite{J09, J10, AS20, J21, RS23, MR23, NP25}. Furthermore, although a fixed left-invariant sub-Riemannian metric may satisfy the measure contraction property $\mathrm{MCP}(0,N)$, a weaker variant of the curvature-dimension condition implying a Bishop–Gromov-type volume comparison (cf. \cite{BR20}). %although for one fixed left-invariant sub-Riemannian metric a weaker variant of the curvature-dimension condition, called measure contraction property (or MCP for short) does hold (cf. \cite{BR20}), which implies Bishop–Gromov volume comparison estimate, 
It is still not clear how we go back to the Riemannian metric. Moreover, even though we can go back, the curvature exponent (the minimal number $N$ such that $\mathrm{MCP}(0,N)$ holds) is strictly larger than $Q$. We refer interested reader to \cite{J09, R13, R16, BR18, GN24, Z25} for the results on the curvature exponent. So we cannot use curvature-dimension conditions to deduce Theorem \ref{t1}. However, as a byproduct of our proof, we can also obtain the following Bishop–Gromov type inequality. 

\begin{proposition}\label{pro1}
Assume $\G$ is a step-two Carnot group.   For any  $g \in \mathcal{L}(\mathbb{G})$, we have
    \[
    \frac{\nu_g(B_g (p,R))}{\nu_g(B_g (p,r))} \le \left( \frac{R}{r} \right)^Q, \qquad \forall \, p \in \G, \ 0<r \le R.
    \]
\end{proposition}

We split the proof of Theorem \ref{t1} into two parts: 1. \(D(\G) \le 2^Q\); 2. \(D(\G) \ge 2^Q\). To prove \(D(\G) \le 2^Q\), we use a key scaling property (see Lemma \ref{lort}) when the rescaling is applied only to the second layer and the two layers in the stratification are orthogonal. Then we will choose suitable subspaces to realize such orthogonality, which constitutes the most significant innovation of the paper, since the natural approach to such a problem is to first fix a stratification and then express the Riemannian metric relative to that stratification. However, on  higher-step groups it is not clear how to choose such subspaces, which is the main obstacle there.  On the other hand, the inequality \(D(\G) \ge 2^Q\) comes from a limiting argument involving a scaling family of Riemannian metrics converging to a sub-Riemannian metric. The homogeneity of the latter gives that the volume doubling constant is exactly \(2^Q\). %Therefore, the doubling behavior of the Riemannian family \(\mathcal L(\G)\) is governed by the sub-Riemannian limit. 
Alternatively, the lower bound \(D(\G)\geq 2^Q\) follows from the recent volume asymptotic result $\nu_g(B_g(p,r)) = c_g r^Q + O(r^{Q-2})$ as $r \to + \infty$ \cite[Theorem 4.4]{DNGR25}.

\medskip

To the best of our knowledge, no uniform doubling result was previously known for the full family of left-invariant Riemannian metrics on an arbitrary step-two Carnot group. In \cite{BL26}, by establishing uniform volume upper and lower bounds, uniform doubling property is obtained for a natural class of metrics on generalized H-type groups, with a doubling estimate of type \(D\leq C2^Q\), see also Remark \ref{rem}. The proof of Theorem \ref{t1}, instead, is more direct and yields the exact constant \(2^Q\).

\medskip

As noted in \cite{EGS18}, the uniform doubling property has several important analytic consequences. To keep the paper as concise as possible and avoid introducing excessive notation, we do not state the precise theorems here; interested readers may refer to \cite[(1)-(5) of Theorem 9.1]{EGS24} for details.

\medskip

This article is organized as follows. In Section \ref{SecPre} we recall basic facts of step-two Carnot groups. In Section \ref{SecPro} we give the proof of Theorem \ref{t1} as well as Proposition \ref{pro1}.

\medskip

\section{Preliminaries}\label{SecPre}

\medskip

\subsection{Step-two Carnot groups}

Recall that a connected and simply connected Lie group $\G = (\G, \cdot)$ is a step-two Carnot group
if its left-invariant Lie algebra $\mathfrak{g}$ admits a stratification
\begin{align}\label{defstr}
\mathfrak{g} = \mathfrak{g}_1 \oplus \mathfrak{g}_2, \quad
[\mathfrak{g}_1, \mathfrak{g}_1] = \mathfrak{g}_2, \quad
[\mathfrak{g}_1, \mathfrak{g}_2] = \{0\},
\end{align}
where $[\cdot,\cdot]$ denotes the Lie bracket on $\mathfrak{g}$. Via the Lie group exponential map, we can regard $(\G, \cdot)$ as $(\g, \ast)$ with the group operation $\ast$ on $\g$ given by
\begin{align}\label{mulalg}
a \ast b := a + b + \frac{1}{2} [a,b],  \qquad \forall \, a,b \in \g.
\end{align}
On $(\g, \ast)$ the identity is $0$ and the inverse element of $a$ is $- a$. Note that the stratification in \eqref{defstr} is not necessarily unique. For more details, we refer to \cite{BLU07}.

\medskip

\subsection{Left-invariant Riemannian structures}

Given an inner product $\langle \cdot, \cdot \rangle$ of $\g$, we can define a left-invariant Riemannian metric $g$ on $T \G$. To be more precise, let  $n := \dim \g$ and assume that $\{\X_1,\ldots, \X_n\}$ is an orthonormal basis of $\g$ under $\langle \cdot, \cdot \rangle$. Then $\{\X_1(p),\ldots, \X_n(p)\}$ forms an orthonormal basis of $T_p \G$ at every point $p \in \G$. Here 
\begin{align}\label{defleft}
\X_j (p) = d L_p (0) \X_j, \qquad \forall \, p \in \G,    
\end{align}
where $L_p(q) := p \cdot q$ and $d$ denotes the differential. For a smooth path $\gamma : [0,1] \to \G$  we can calculate its length  by: 
\begin{equation}
\label{lenght}
\ell_g(\gamma) := \int_0^1 \sqrt{g_{\gamma(\tau)}(\dot{\gamma}(\tau),\dot{\gamma}(\tau))} \, d\tau.
\end{equation}
Then the Riemannian distance between two points $p, q \in \G$ w.r.t. $g$ is defined as
\[
d_g(p,q) := \inf\{\ell_g(\gamma): \gamma \mbox{ smooth}, \gamma(0) = p, \gamma(1) = q\}.
\]
From the definition it is clear that if $g \le g'$ we have
\begin{align}\label{compar}
    d_g(p,q) \le d_{g'}(p,q), \qquad \forall \, p,q \in \G.
\end{align}

\medskip

\subsection{Dilations}

In this subsection we fix a stratification satisfying \eqref{defstr}. We define  for every $\lambda > 0$, the dilation $\delta_\lambda$ on the Lie algebra is the linear map such that $\delta_\lambda a = \lambda^\ell a$ for $a \in \g_\ell$, $\ell = 1,2$. It is actually a Lie algebra automorphism so it induces the dilation on $\G$, which is a Lie group automorphism and we still denote it by $\delta_\lambda$. We choose $\mu$ to be the pushforward measure of the Lebesgue measure on $\g$ by the Lie group exponential map (so it is a Haar measure on $\G$). As a result, we have 
\begin{align}\label{Haar}
\mu(\delta_\lambda(A)) = \lambda^Q \mu(A), \qquad \forall \, A \subset \G \ \mbox{measurable}, \lambda > 0,
\end{align}
with homogeneous dimension 
\begin{equation}\label{hd}
    Q: = \dim \g_1 + 2 \dim \g_2,
\end{equation}
 which is strictly larger than the topological dimension  $n = \dim \g$. 

\medskip

\section{Proof of Theorem \ref{t1}}\label{SecPro}

\medskip

\begin{lemma}\label{lort}
Assume that the left-invariant Riemannian metric $g$ is induced by an inner  product $\langle \cdot, \cdot \rangle$ on $\g$ such that the subspaces  $\g_1$ and $\g_2$ in the stratification \eqref{defstr} are orthogonal. We decompose 
\[
\langle \cdot, \cdot \rangle = \langle \cdot, \cdot \rangle_{\g_1} + \langle \cdot, \cdot \rangle_{\g_2}
\]
and for $\lambda > 0$ define
\[
\langle \cdot, \cdot \rangle_\lambda := \langle \cdot, \cdot \rangle_{\g_1} + \lambda \langle \cdot, \cdot \rangle_{\g_2}.
\]
If we use the notation $g_\lambda$ to denote the corresponding left-invariant Riemannian metric induced by $\langle \cdot, \cdot \rangle_\lambda$, then the following relation holds:
\begin{align}\label{regg}
\lambda d_g(p,q) =  d_{g_{\lambda^{-2}}}(\delta_\lambda(p), \delta_\lambda(q)) , \qquad \forall \, p, q \in \G, \lambda > 0.
\end{align}
\end{lemma}

\begin{proof}
Throughout the proof we regard $(\G,\cdot)$ as $(\g,\ast)$. For any smooth path $\gamma: [0,1] \to \G$ joining $p$ with $q$, we can write 
\begin{equation}\label{CauchyProblem}
\dot{\gamma}(s) = \sum_{j = 1}^n u_j(s) \X_j(\gamma(s)), \qquad \gamma(0) = p,  \qquad \gamma(1) = q,
\end{equation}
where $\{\X_1,\ldots, \X_m\}$ is an orthonormal basis of $\g_1$ under $\langle \cdot, \cdot \rangle_{\g_1}$, $\{\X_{m + 1},\ldots, \X_n\}$ is an orthonormal basis of $\g_2$ under $\langle \cdot, \cdot \rangle_{\g_2}$, and $u_1(\cdot), \ldots, u_n(\cdot)$ are smooth functions. From the formula \eqref{mulalg}, we obtain
\[
d L_p (0) a =  a + \frac{1}{2} [p,a], \qquad \forall \, p,a \in \g
\]
and combining this with  \eqref{defleft}, we get  
\begin{align}\label{curve}
\dot{\gamma}(s) = \sum_{j = 1}^m u_j(s) \X_j + \frac{1}{2} \left[\gamma(s), \sum_{j = 1}^m u_j(s) \X_j\right] + \sum_{j = m +1}^n u_j(s) \X_j, \quad \forall \, s \in [0,1].
\end{align}
Then for any $\lambda > 0$, define $\gamma_\lambda := \delta_\lambda(\gamma)$. Factoring $\delta_\lambda$ on both sides of \eqref{curve}, we obtain
\[
\dot{\gamma}_\lambda(s) =  \sum_{j = 1}^m \lambda u_j(s) \X_j(\gamma_\lambda(s))  + \sum_{j = m + 1}^n \lambda^2 u_j(s) \X_j(\gamma_\lambda(s)), \quad \gamma_\lambda(0) = \delta_\lambda (p), \quad \gamma_\lambda(1) = \delta_\lambda(q).
\]
So $\gamma_\lambda$ is a smooth curve joining $\delta_\lambda(p)$ with $\delta_\lambda(q)$ and $\ell_{g_{\lambda^{-2}}}(\gamma_\lambda) = \lambda\ell_g(\gamma)$, which implies
\[
d_{g_{\lambda^{-2}}}(\delta_\lambda(p), \delta_\lambda(q)) \le \lambda d_g(p,q), \qquad \forall \, p,q \in \G, \lambda > 0.
\]
Since in the argument above the inner product $g$, points $p,q \in \G$, and $\lambda > 0$ are arbitrary, we obtain
\[
d_g(p,q) \le \lambda d_{g_{\lambda^2}} (\delta_{\lambda^{-1}}(p),\delta_{\lambda^{-1}}(q)), \qquad \forall \, p,q \in \G, \lambda > 0.
\]
Replacing $\lambda$ with $\lambda^{-1}$, we can prove the opposite direction and this proves the lemma.
\end{proof}

\medskip

\begin{proof}[Proof of Theorem \ref{t1}]
Given a step-two Carnot group $\G$ and a left-invariant Riemannian metric $g$ induced by the inner product $\langle \cdot, \cdot \rangle$, we define $\g_2 := [\g,\g]$ and $\g_1$ is the orthogonal complement of $\g_2$ in $\g$ w.r.t. $\langle \cdot, \cdot \rangle$. We can prove that 
\begin{align}\label{straii}
\g = \g_1 \oplus \g_2, \qquad \g_2 = [\g,\g] = [\g_1,\g_1], \qquad [\g_1,\g_2] = [\g_1,[\g,\g]] = \{0\}.
\end{align}
So this decomposition gives a stratification. Since $\nu_g$ is left-invariant and thus is a Haar measure, we obtain $\nu_g = c_g\mu$ for some $c_g > 0$. Using the left-invariance, and the result in Lemma \ref{lort} (with $\lambda = 1/2$ and $p = 0$), we obtain
\[
\frac{\mu(B_g(p,2r))}{\mu(B_g(p,r))} = \frac{\mu(B_g(0,2r))}{\mu(B_g(0,r))} =  \frac{\mu(\delta_2(B_{g_4}(0,r)))}{\mu(B_g(0,r))}  = 2^Q  \frac{\mu(B_{g_4}(0,r))}{\mu(B_g(0,r))} \le 2^Q,
\]
where we have also used \eqref{Haar} and \eqref{compar} respectively. Since $g$ is arbitrary, we proved that 
\[
\sup_{g \in \mathcal{L}(\G), \, p \in \G, \, r > 0} \frac{\nu_g(B_g(p,2r))}{\nu_g(B_g(p,r))}  \le 2^Q.
\]
To show the opposite inequality, for the left-invariant Riemannian metric $g$ in the beginning of the proof with the stratification \eqref{straii}, we can define a family of left-invariant Riemannian metrics $\{g_\lambda\}_{\lambda > 0}$ as in Lemma \ref{lort}. It follows from \cite[Theorem 1.4 (see also Remark 4.2 in \S~4.1 there)]{ALN23} that $d_{g_\lambda} \to d_{g_\infty}$ as $\lambda \to +\infty$ uniformly on compact subsets of $\G \times \G$. Here $d_{g_\infty}$ is a finite (sub-Riemannian) distance on $\G$. Under the limit, \eqref{regg} becomes
\begin{align}\label{regg2}
\lambda d_{g_\infty}(p,q) =  d_{g_\infty}(\delta_\lambda(p), \delta_\lambda(q)) , \qquad \forall \, p, q \in \G, \lambda > 0.
\end{align}
Since $\lambda \le \lambda'$ implies
$d_{g_\lambda} \le d_{g_{\lambda'}}$ by \eqref{compar}, we have 
\[
B_{g_{\lambda'}}(0,r) \subset B_{g_{\lambda}}(0,r).
\]
As a result, we obtain
\[
B_{g_\infty}(0,r) \subset \cap_{\lambda > 0} B_{g_{\lambda}}(0,r) = \lim_{\lambda \to +\infty} B_{g_{\lambda}}(0,r) \subset  \overline{B_{g_\infty}(0,r)}.
\]
Now \eqref{regg2} implies that the sphere has measure $0$ and so we have
\[
\lim_{\lambda \to +\infty} \mu(B_{g_{\lambda}}(0,r)) = \mu(B_{g_\infty}(0,r)), \qquad \forall \, r > 0.
\]
By \eqref{regg2} again, we have for any $r > 0$
\[
2^Q = \frac{\mu(B_{g_\infty}(0,2r))}{\mu(B_{g_\infty}(0,r))} = \lim_{\lambda \to +\infty} \frac{\mu(B_{g_{\lambda}}(0,2r))}{\mu(B_{g_{\lambda}}(0,r))} \le \sup_{g \in \mathcal{L}(\G), \, p \in \G, \, r > 0} \frac{\nu_g(B_g(p,2r))}{\nu_g(B_g(p,r))}.
\]
\end{proof}

\begin{proof}[Proof of Proposition \ref{pro1}]
It is sufficient to follow the argument in the proof of Theorem \ref{t1} but with $\lambda = \frac{r}{R}$ instead of $1/2$.
\end{proof}

\begin{remark}\label{rem}
The Heisenberg-type groups $\mathbb{H}(2n,m)$ has homogeneous dimension $Q = 2n+2m$. Let $\{\mathrm{X}_1, \ldots, \mathrm{X}_{2n}, \mathrm{T}_1,\ldots,\mathrm{T}_m \}$ be the canonical basis, in \cite{BLZ25} the authors considered the uniform volume estimates for a natural class of Laplacian $\Delta_\mathbb{B} = \sum_{j = 1}^{2n} \mathrm{X}_j^2 + \sum_{i = 1}^m (\sum_{j = 1}^m \mathbb{B}_{i,j} \mathrm{T}_j)^2$, where $ (\mathbb{B}_{i,j})_{1 \le i,j \le m}= \mathbb{B}$ is an arbitrary real matrix. Let $B_{\mathbb{B}}(p,r)$ be the ball centered at $p$ with radius $r>0$ defined by the Carnot-Carath\'{e}odory distance associated to $\Delta_\mathbb{B}$. One can deduce from the proof of \cite[Lemma 5.3]{BLZ25} that
\[
\mu (B_{\mathbb{B}}(0,r)) = \frac{\pi^{\frac{2n+m}{2}}}{\Gamma(n) n^{\frac{m}{2}+1}} \, r^{2n+m} \det \left( \frac{r^2}{12} \mathbb{I}_m + \mathbb{B}\mathbb{B}^T  \right)^{\frac{1}{2}} (1+ o_n (1) ),
\]
which holds uniformly in $(\mathbb{B}, r, m)$ as $n \to +\infty$. This gives
\[
\sup_{\mathbb{B},\, r > 0} \frac{\mu(B_{\mathbb{B}}(0,2r))}{\mu(B_\mathbb{B}(0,r))} = 2^Q \, (1+ o_n(1)).
\]
From Theorem \ref{t1}, we know the doubling constant here is indeed $2^Q$ by a possible approximation argument like the one in the proof of Theorem \ref{t1} (since the matrix $\mathbb{B}$ is not necessarily invertible).
    
\end{remark}

\medskip

\section*{Acknowledgements}

\medskip

This project has received funding from the European Research Council (ERC) under the European Union’s Horizon 2020 research and innovation programme (grant agreement GEOSUB, No. 945655).

%\bigskip

%\nocite{*}
%\bibliographystyle{alpha}
%\bibliography{UDbib}

\begin{thebibliography}{ALDNG23}

\bibitem[ALDNG23]{ALN23}
Gioacchino Antonelli, Enrico Le~Donne, and Sebastiano Nicolussi~Golo.
\newblock Lipschitz {C}arnot-{C}arath\'eodory structures and their limits.
\newblock {\em J. Dyn. Control Syst.}, 29(3):805--854, 2023.

\bibitem[AS20]{AS20}
Luigi Ambrosio and Giorgio Stefani.
\newblock Heat and entropy flows in {C}arnot groups.
\newblock {\em Rev. Mat. Iberoam.}, 36(1):257--290, 2020.

\bibitem[BL26]{BL26}
Cheng {Bi} and Hong-Quan {Li}.
\newblock {Uniform volume estimates and maximal functions on generalized
  Heisenberg-type groups}.
\newblock {\em arXiv e-prints}, page arXiv:2604.14715, April 2026.

\bibitem[BLU07]{BLU07}
A.~Bonfiglioli, E.~Lanconelli, and F.~Uguzzoni.
\newblock {\em Stratified {L}ie groups and potential theory for their
  sub-{L}aplacians}.
\newblock Springer Monographs in Mathematics. Springer, Berlin, 2007.

\bibitem[BLZ25]{BLZ25}
Cheng {Bi}, Hong-Quan Li, and Ye~Zhang.
\newblock Centered {H}ardy-{L}ittlewood maximal functions on {H}-type groups
  revisited.
\newblock {\em Math. Ann.}, 391(3):3765--3797, 2025.

\bibitem[BR18]{BR18}
Davide Barilari and Luca Rizzi.
\newblock Sharp measure contraction property for generalized {H}-type {C}arnot
  groups.
\newblock {\em Commun. Contemp. Math.}, 20(6):1750081, 24, 2018.

\bibitem[BR20]{BR20}
Zeinab Badreddine and Ludovic Rifford.
\newblock Measure contraction properties for two-step analytic sub-{R}iemannian
  structures and {L}ipschitz {C}arnot groups.
\newblock {\em Ann. Inst. Fourier (Grenoble)}, 70(6):2303--2330, 2020.

\bibitem[EGSC18]{EGS18}
Nathaniel Eldredge, Maria Gordina, and Laurent Saloff-Coste.
\newblock Left-invariant geometries on {$\rm SU(2)$} are uniformly doubling.
\newblock {\em Geom. Funct. Anal.}, 28(5):1321--1367, 2018.

\bibitem[EGSC24]{EGS24}
Nathaniel {Eldredge}, Maria {Gordina}, and Laurent Saloff-Coste.
\newblock {Uniform doubling for abelian products with $\operatorname{SU}(2)$}.
\newblock {\em arXiv e-prints}, page arXiv:2412.17102, December 2024.

\bibitem[GZ24]{GN24}
Sebastiano~Nicolussi Golo and Ye~Zhang.
\newblock Curvature exponent and geodesic dimension on {S}ard-regular {C}arnot
  groups.
\newblock {\em Anal. Geom. Metr. Spaces}, 12(1):Paper No. 20240004, 30, 2024.

\bibitem[Jui09]{J09}
Nicolas Juillet.
\newblock Geometric inequalities and generalized {R}icci bounds in the
  {H}eisenberg group.
\newblock {\em Int. Math. Res. Not. IMRN}, (13):2347--2373, 2009.

\bibitem[Jui10]{J10}
Nicolas Juillet.
\newblock On a method to disprove generalized {B}runn-{M}inkowski inequalities.
\newblock In {\em Probabilistic approach to geometry}, volume~57 of {\em Adv.
  Stud. Pure Math.}, pages 189--198. Math. Soc. Japan, Tokyo, 2010.

\bibitem[Jui21]{J21}
Nicolas Juillet.
\newblock Sub-{R}iemannian structures do not satisfy {R}iemannian
  {B}runn-{M}inkowski inequalities.
\newblock {\em Rev. Mat. Iberoam.}, 37(1):177--188, 2021.

\bibitem[LNNR25]{DNGR25}
Enrico {Le Donne}, Luca {Nalon}, Sebastiano {Nicolussi Golo}, and Seung-Yeon
  {Ryoo}.
\newblock {Asymptotics of Riemannian Lie groups with nilpotency step 2}.
\newblock {\em arXiv e-prints}, page arXiv:2503.00560, March 2025.

\bibitem[MR23]{MR23}
Mattia Magnabosco and Tommaso Rossi.
\newblock Almost-{R}iemannian manifolds do not satisfy the curvature-dimension
  condition.
\newblock {\em Calc. Var. Partial Differential Equations}, 62(4):Paper No. 123,
  27, 2023.

\bibitem[NP26]{NP25}
Dimitri Navarro and Jiayin Pan.
\newblock Universal non-cd of sub-riemannian manifolds.
\newblock {\em Journal für die reine und angewandte Mathematik (Crelles
  Journal)}, 2026.

\bibitem[Rif13]{R13}
Ludovic Rifford.
\newblock Ricci curvatures in {C}arnot groups.
\newblock {\em Math. Control Relat. Fields}, 3(4):467--487, 2013.

\bibitem[Riz16]{R16}
Luca Rizzi.
\newblock Measure contraction properties of {C}arnot groups.
\newblock {\em Calc. Var. Partial Differential Equations}, 55(3):Art. 60, 20,
  2016.

\bibitem[RS23]{RS23}
Luca Rizzi and Giorgio Stefani.
\newblock Failure of curvature-dimension conditions on sub-{R}iemannian
  manifolds via tangent isometries.
\newblock {\em J. Funct. Anal.}, 285(9):Paper No. 110099, 31, 2023.

\bibitem[Zha25]{Z25}
Ye~Zhang.
\newblock On the lower bound of the curvature exponent on step-two {C}arnot
  groups.
\newblock {\em Proc. Amer. Math. Soc.}, 153(10):4437--4446, 2025.

\end{thebibliography}

\end{document}